\documentclass{article}
\usepackage[utf8]{inputenc}
\usepackage{amsmath, amssymb, enumerate, multicol, chngcntr}
\usepackage{amsthm}
\usepackage{xcolor}
\usepackage{pdfpages}
\usepackage{comment}
\usepackage{caption}
\usepackage{bm}
\usepackage{mathtools}

\theoremstyle{plain}
\newtheorem{thm}{Theorem}[section]
\newtheorem{cor}[thm]{Corollary}
\newtheorem{prop}[thm]{Proposition}

\newtheorem{rem}{Remark}

\counterwithin{equation}{section}

\usepackage{fancyhdr}
\usepackage{setspace}

\usepackage[english]{babel}
\usepackage[euler]{textgreek}
\usepackage{stmaryrd}

\usepackage{tikz}
\usetikzlibrary{shapes,arrows}
\usepackage{tikz-cd}
\usepackage{amssymb}
\usepackage{graphicx}
\graphicspath{{Images/}}
\usepackage[font=small,labelfont=bf]{caption}
\usepackage{pgf,tikz,pgfplots}
\pgfplotsset{compat=1.15}
\usepackage{mathrsfs}
\usepackage{diagbox}
\usepackage{float}
\usepackage{xcolor}
\usepackage{ragged2e}
\usepackage{enumerate}
\usepackage{upgreek}
\usepackage{multicol}

\theoremstyle{remark}

\theoremstyle{definition}
\newtheorem{definition}[thm]{Definition}

\theoremstyle{plain}

\usepackage[hidelinks]{hyperref}
\usepackage{bookmark}
\usepackage{cleveref}

\hypersetup{
	colorlinks   = true, 
	urlcolor     = blue, 
	linkcolor    = blue, 
	citecolor   = red, 
	citecolor   = blue,
}

\newtheoremstyle{note}
{3pt}
{3pt}
{}
{}
{\itshape}
{:}
{.5em}
{}

\newtheoremstyle{citing}
{3pt}
{3pt}
{\itshape}
{}
{\bfseries}
{.}
{.5em}
{\thmnote{#3}}

\theoremstyle{citing}

\newtheoremstyle{break}
{9pt}
{9pt}
{\itshape}
{}
{\bfseries}
{.}
{\newline}
{}

\let\lvert=|\let\rvert=|

\title{On the normal subgroups of a split extension}
\author{Prashun Kumar \footnote{Corresponding author, Dr. B. R. Ambedkar University Delhi, Delhi 110006, India; \ E-mail:  prashun07kumar@gmail.com.}}

\begin{document}
	\fontfamily{cmr}\selectfont
	
	\maketitle
	
	\bigskip
	\noindent
	{\small{\bf ABSTRACT:}} Let $N$ and $Q$ be a finite groups with $\gcd(|N|,|Q|) = 1$. In this paper we describe normal subgroups of $G =  N \rtimes Q$ via normal subgroups of $N$ and $Q$. Let $p$ and $q$ be distinct primes. Let $\mathfrak{A}_p$ be the variety of  elementary abelian $p$-groups. Let $\mathfrak{A}_p\mathfrak{A}_q$ be the variety of extensions of a group in $\mathfrak{A}_p$ by a group in $\mathfrak{A}_q$. We also provide a method of determining the normal subgroups of a group in the variety $\mathfrak{A}_p\mathfrak{A}_q$. We also provide the complete list of normal subgroups of a finite group with cyclic Sylow subgroups.
	
	
	\medskip
	\noindent
	{\small{\bf Mathematics Subject Classification-MSC2020}{:} }
	20E28, 20E34, 20E45, 20F99. 
	
	\baselineskip=\normalbaselineskip
	
	\section{Introduction}
	V. M. Usenko in 1991 described the subgroups of a given split extension and he also provided the necessary condition for a subgroup of a split extension $G$ to be normal in $G$ (see \cite{U1991}). In this paper we consider a split extension of the form $G = N \rtimes Q$ where $N$ and  $Q$ are finite groups  of co-prime order and describe the normal subgroups of $G$ in terms of normal subgroups of $N$ and $Q$. Additionally as an application of our main result we provide complete list of normal subgroups of a finite group with cyclic Sylow subgroups. We refer to Dietrich-Low \cite{DL2021} for a recent discussion on groups with cyclic Sylow subgroups. Since Dietrich-Low used the term $C$-group for a finite group having cyclic Sylow subgroups we also call a finite group with cyclic Sylow subgroups a $C$-group.
	
	 Let $\mathfrak{A}_p$ and $\mathfrak{A}_q$ be the varieties consisting of the elementary abelian $p$- and $q$-groups where $p$ and $q$ are distinct primes. Let $\mathfrak{A}_p\mathfrak{A}_q$ be the variety consisting of extensions of a group in $\mathfrak{A}_p$ by a group in $\mathfrak{A}_q$. In this paper we provide a method to determine the normal subgroups of a group in the variety $\mathfrak{A}_p\mathfrak{A}_q$. Further we also provide complete list of normal subgroups of a group in the variety $\mathfrak{A}_p\mathfrak{A}_q$ in some special cases. For the background on the theory of varieties of groups we refer to the classical book by H. Neumann (see \cite{HN1967}). 
	
	\medskip
	Throughout the article $p$ and $q$ are distinct primes. ${\rm GL}(n,p)$ is the group of $n\times n$ non-singular matrices over the finite field of order $p$. Let $G$ be a group and let $a,x \in G$. Then we denote the conjugate $xax^{-1}$ as $a^x$ and the commutator $axa^{-1}x^{-1}$ as $[a,x]$.  Let $S$ and $T$ be subsets of $G$. Then we denote $[S,T]$ the subgroup of $G$ generated by the set $\{[a,x]\mid a \in S $ and $x \in T\}$. Let $K$ be a subgroup of $G$. Then $C_K(T) = C_G(T)\cap K$. We write $K \leq_{char} G$ if $K$ is  a characteristic subgroup of $G$. By transversal of $K$ in $G$, we always mean right transversal of $K$ in $G$. 
	
	\section{Normal Subgroups of a split extension}
	
	In this section we prove some results concerning normal subgroups of a split extensions and prove our main result. 
	
	\medskip
	We begin by providing necessary and sufficient conditions for a subgroup of certain type to be normal in a split extension.
	
	\begin{prop}{\label{Normal_subgroups}}
		Let $G = N \rtimes Q$. Let $S = \hat{N} \rtimes \hat{Q}$ be the subgroup of $G$ where $\hat{N} \leq N$ and $\hat{Q} \leq Q$. Then $S$ is normal in $G$ if and only if $\hat{N} \trianglelefteq N$, $\hat{Q} \trianglelefteq Q$, $[N,\hat{Q}] \leq \hat{N}$ and $[\hat{N},Q] \leq \hat{N}$. 
	\end{prop}
	
	\begin{proof}
		Suppose the $S$ is normal in $G$. Then for $a \in N$ we have $\hat{N}^{a} \leq S$. Further $\hat{N}^{a}\leq N$ and every element of $\hat{N}^a$ has a unique representation in $G$. Therefore $\hat{N}^a = \hat{N}$ and $\hat{N}\trianglelefteq N$. Similarly it can be shown that $\hat{N}$ is $Q$-invariant. Therefore $\hat{N}$ is normal in $G$ and $[\hat{N},Q] \leq \hat{N}$. Using similar line of arguments it can be shown that $\hat{Q} \trianglelefteq Q$. Also for $x \in N$ we have $\hat{Q}^x \leq S$. Again every element of $\hat{Q}^x$ has a unique representation in $G$ and it can be shown that $[N,\hat{Q}] \leq \hat{N}$. Conversely the conditions $\hat{N}\trianglelefteq N$ and $[\hat{N},Q] \leq \hat{N}$ together imply that $\hat{N}$ is normal in $G$. Further the conditions $\hat{Q} \trianglelefteq Q$ and $[N,\hat{Q}]\leq \hat{N}$ imply that all conjugates of $\hat{Q}$ are contained in $S$.   
	\end{proof}
	
	The conditions $\hat{N} \trianglelefteq N$ and $[\hat{N},Q]$ in Proposition \ref{Normal_subgroups} can be combined into one $N \trianglelefteq G$. This observation leads to the following corollary. 
	
	\begin{cor}{\label{cor_to_normal_subgroups}}
		Let $G = N \rtimes Q$. Let $S = \hat{N} \rtimes \hat{Q}$ be the subgroup of $G$ where $\hat{N} \leq N$ and $\hat{Q} \leq Q$. Then $S$ is normal in $G$ if and only if $\hat{N} \trianglelefteq G$, $\hat{Q} \trianglelefteq Q$ and $[\hat{N},Q] \leq \hat{N}$. In particular $[N,Q] \rtimes \hat{Q}$ is normal in $G$ for any normal subgroup $\hat{Q}$ of $Q$.
	\end{cor}
	
	\begin{proof}
		Since the conditions $\hat{N}\trianglelefteq N$ and $[\hat{N},Q] \leq \hat{N}$ hold if and only if $\hat{N}$ is normal in $G$ the result follows from Proposition \ref{Normal_subgroups}.
	\end{proof}

	Now we state and prove our main result. 
	 
     \begin{thm}{\rm (\bf Main Result)}{\label{normal_subgroup_of_a_split_extension}}
		Let $N$ and $Q$ be finite groups such that $\gcd(|N|,|Q|)=1$. Let $G = N \rtimes Q$ and let $S$ be a normal subgroup of $G$. Then one of the following holds.
		
		\begin{enumerate}[$(\rm i)$]
			\item $S$ is a $Q$-invariant normal subgroup of $N$.
 			\item $S$ is a subgroup of $C_Q(N)$.
			\item $S =  \hat{N} \rtimes \hat{Q}$ where $\hat{N} \leq N$ is normal in $G$ and contains a transversal of the subgroup $C_N(\hat{Q})$ in $N$ and $\hat{Q}$ is a normal subgroup of $Q$. 
		\end{enumerate}
		Conversely the subgroups of $G$ described in $(\rm i)$, $(\rm ii)$ and $(\rm iii)$ are normal in $G$. 
	\end{thm}
	
	\begin{proof}
		Suppose that $|S| \mid |N|$. Then $S$ is contained in $N$. Since $S$ is a normal subgroup of $G$, $S$ is a $Q$-invariant normal subgroup of $N$.
		
		\medskip
		Now suppose that $|S|\mid |Q|$. Then by Schur-Zassenhaus Theorem $S \leq Q$. Further $S$ acts trivially on $N$ and therefore  $S$ is a subgroup of $C_Q(N)$.
		
		\medskip
		Finally suppose that $|S|= mn$ where $m \mid |N|$ and $n \mid |Q|$. Then $S \cap N$ is the normal Hall $n'$-subgroup of $S$. Therefore by Schur-Zassenhaus Theorem $S = \hat{N} \rtimes \hat{Q}$ where $\hat{N} = S \cap N$ and $\hat{Q}$ is a subgroup of $G$ contained in a conjugate of $Q$. Since $S$ is normal in $G$ we can assume that $\hat{Q} \leq Q$. Let $T$ be a transversal of $C_N(\hat{Q})$ in $N$ and let $t \in T$. Since $\hat{Q}^t \leq S$ by Schur-Zassenhaus Theorem, $\hat{Q}^t$ is conjugate to $\hat{Q}$ in $S$. Therefore $\hat{Q}^t = \hat{Q}^{\hat{a}}$ for some $\hat{a} \in \hat{N}$. But then for any $y \in \hat{Q}$ we have $[t^{-1}\hat{a},y] \in N \cap \hat{Q} = \{1\}$. Thus $\hat{a} = tb$ for some $b \in C_N(\hat{Q})$. So $\hat{a} \in \hat{N}$ is a representative of the coset $tC_N(\hat{Q})$ in $N$ and hence $\hat{N}$ contains a transversal of $C_N(\hat{Q})$ in $N$.
        
        \medskip		
		
		 Conversely, since $\hat{N}$ is normal in $G$ we have $[\hat{N},Q] \leq \hat{N}$. Further if $\hat{N}$ contains transversal $T$ of $C_N(Q)$ in $N$, then $ [N,\hat{Q}] = [T,\hat{Q}] \leq  \hat{N}$. Therefore by Proposition \ref{Normal_subgroups}, $S$ is normal in $G$. 
	\end{proof}
	    
	    Assume the notation of Theorem \ref{normal_subgroup_of_a_split_extension}. As it can be seen from the statement of Theorem \ref{normal_subgroup_of_a_split_extension} that the transversal of centralizer of a normal subgroup of $Q$ plays an important role in determining the normal subgroups of $G$ of the form $\hat{P} \rtimes \hat{Q}$ where $\hat{P} \leq P$ and $\hat{Q} \leq Q$.
	    
	    \medskip
		Let $N$ and $H$ be arbitrary finite groups. Let $\hat{H}$ be a normal subgroup of $H$.  The following result provides a relation between the transversal of the centralizer $C_N(\hat{H})$ in $N$ and the conjugates of $\hat{H}$ in $N \rtimes H$. In particular  we provide the complete and irredundant list of conjugates of a normal subgroup $\hat{H}$ of $H$ in $N \rtimes H$. We also provide the structure of the normalizer of  $\hat{H}$ in $N \rtimes H$. 
	
	\begin{prop}{\label{Complete_conjugates}}
		Let $N$ and $H$ be finite groups. Let $G = N \rtimes H$ and let $\hat{H}$ be a normal subgroup of $H$. Let $T$ be a transversal of $C_N(\hat{H})$ in $N$. Then
		
		\begin{enumerate}[$(\rm i)$]
			\item the set ${\cal L} = \{\hat{H}^t\mid t \in T\}$ represents the complete and irredundant list of conjugates of $\hat{H}$ in $G$.
			\item  $N_G(\hat{H}) = C_N(\hat{H}) \rtimes H$.
		\end{enumerate}
	\end{prop}
	
	\begin{proof}
		Let $g = ax$ where $a \in N$ and $x \in H$. Then $\hat{H}^g = \hat{H}^a$. Further we can write $a = ta'$ where $a' \in C_N(\hat{H})$ and $t \in T$, therefore $\hat{H}^g = \hat{H}^t \in {\cal L}$. Now if $\hat{H}^t = \hat{H}^{t'}$ for some $t,t' \in T$ then for any $x \in \hat{H}$ we have $[t^{-1}t',x] \in N\cap \hat{H} = \{1\}$. Therefore $t' = tb$ for some $b \in C_N(\hat{H})$. But $t$ and $t'$ are elements of transversal $T$ so we must have $t = t'$. 
		
		\medskip 
		Clearly $C_N(\hat{H}) = C_G(\hat{H})\cap N$ is normal in $N_G(\hat{H})$. Further since $\hat{H}$ is normal in $H$ we have $H \leq N_G(\hat{H})$. So $C_N(\hat{H})H \leq N_G(\hat{H})$. Since $|T|= |N: C_N(\hat{H})|$ is equal to the number of conjugates of $H$ in $G$ we have that $|N: C_N(\hat{H})| = |G:N_G(\hat{H})|$. Thus $|N_G(\hat{H})| = |C_N(\hat{H})| |H|$ and the result follows.  
	\end{proof}
	
	Proposition \ref{Complete_conjugates} allows us to provide the necessary condition for the normality of a certain type of subgroups of $G$. That is we have the following result.
	
     \begin{cor}
     	Assume the notation of Proposition \ref{Complete_conjugates}. Let $\hat{N}$ be a subgroup of $N$ that is normal in $G$. If $\hat{N}$ contains a transversal of $C_N(\hat{H})$ in $N$ then $S = \hat{N} \rtimes \hat{H}$ is normal in $G$. 
     \end{cor}
     
     \begin{proof}
     	By proposition \ref{Complete_conjugates}, $S$ contains all conjugates of $\hat{H}$ in $G$ therefore $S$ is normal in $G$.
     \end{proof}
		
		
	\section{Normal subgroup of a finite group with cyclic Sylow subgroups}
	
	Result of H$\ddot{\text{o}}$lder, Burnside and Zassenhaus \cite[Theorem 10.1.10]{R1995} shows that every $C$-group $G$ of order $n$ is metacyclic with odd-order derived subgroup $G' \cong \mathbb{Z}_m$ and cyclic quotient $G/G'$ of order $l = n/m$; specifically, $G$ is isomorphic to 
	$$\langle a,x \mid a^m = 1 =x^l,a^x = a^r \rangle$$ 
	for some $0 \leq r \leq m$ with $r^l \equiv 1 \bmod{m}$ and $\gcd(l(r-1),m) = 1$. Conversely in a group with such a presentation all Sylow subgroups are cyclic.
	
	That is $G \cong {\mathbb{Z}_m \rtimes \mathbb{Z}_l}$ where $\gcd(m,l) = 1$ and $\mathbb{Z}_l$ acts on ${\mathbb{Z}_m}$ by $r^{th}$-powering for some $0 \leq r \leq m$ with $r^l \equiv 1 \bmod{m}$ and $\gcd(r-1,m) = 1$.
	
	\medskip
	Note that $G$ is abelian (cyclic) if and only if $m = 1$ and $r = 0$.
	
	\medskip
	In this section we provide the complete list of normal subgroups of a finite group with cyclic Sylow subgroups as described in the following result. 
	
	\begin{thm}
		Let $G = \langle a,x \mid a^m = 1 =x^l,a^x = a^r \rangle$ be a $C$-group and let $N$ be a normal subgroup of $G$. Then one of the following holds.
		
		\begin{enumerate}[$(\rm i)$]
			\item $N \leq \langle a \rangle$.
			\item $N \leq \langle x^k \rangle$ where $k = {\rm ord}(r\bmod{m})$.
			\item $N =\langle a^{m_1} \rangle \rtimes \langle x^{l_1}\rangle$ where $m_1 \mid s$ and $s \equiv (r^{l_1} - 1)\bmod{m}$.
		\end{enumerate}
		Conversely if $N$ satisfies one of $(\rm i)$, $(\rm ii)$ or $(\rm iii)$, then $N$ is normal in $G$. 
	\end{thm}
	
	\begin{proof}Let us assume that $P = \langle a \rangle$ and $Q = \langle x \rangle$. By Theorem \ref{normal_subgroup_of_a_split_extension}, if $|N|\mid m$, then $N \leq \langle a \rangle$. 
		
		\medskip
		If $|N| \mid l$ then $N \leq C_P(Q)$. Now clearly $\langle x^k \rangle \leq C_Q(P)$. Let $x^u \in C_Q(P)$. Then $x^uax^{-u} = a$ which gives $r^u \equiv 1\bmod{m}$. Thus $k \mid u$ and $x^u \in \langle x^k \rangle$. Whence $C_Q(P) = \langle x^k \rangle$.
		
		\medskip
		Finally if $|N| = m'l'$ where $m'\mid m$ and $l'\mid l$ then $N =\langle a^{m_1} \rangle \rtimes \langle x^{l_1}\rangle$ where $\langle a^{m_1} \rangle$ contains transversal $T$ of $C_P(\langle x^{l_1}\rangle)$ in $P$. So $\langle a^{m_1} \rangle$ contains $[P,\langle x^{l_1}\rangle]$. In particular $\langle a^{m_1} \rangle$ contains $[a^{-1},x^{l_1}] = a^{r^{l_1} -1} = a^{s}$ where $s \equiv (r^{l_1} - 1) \bmod{m}$ and it follows that $m_1 \mid s$. 
		
		Conversely it can be seen that the subgroups listed in part $(\rm i)$, $(\rm ii)$ are normal in $G$. Let $N =\langle a^{m_1} \rangle \rtimes \langle x^{l_1}\rangle$ where $m_1 \mid s$ and $s \equiv (r^{l_1} - 1)\bmod{m}$. Since $[\langle a \rangle, \langle x^{l_1}\rangle] = \langle a^s \rangle \leq N$ by Lemma \ref{Normal_subgroups}, $N$ is normal in $G$.
	\end{proof}

	\section{Normal subgroups of a group in the variety $\mathfrak{A}_p\mathfrak{A}_q$}
	
	Let $G \in \mathfrak{A}_p\mathfrak{A}_q$. Then clearly  $G = P \rtimes Q$ where $P$ is elementary abelian $p$-group and $Q$ is elementary abelian $q$-group. In this section we provide a method to determine normal subgroups of $G$.
	
	\medskip
	Throughout this section we assume that $P$ is an elementary abelian group of order $p^{\alpha}$ and $Q$ is an elementary abelian group of order $q^{\beta}$. 
	
	\medskip
	\noindent
	 By Theorem \ref{normal_subgroup_of_a_split_extension} $(\rm iii)$, the normal subgroups of a split extension of  $P$ by $Q$ depend on the centralizer of the subgroups of $Q$ in $P$. Hence our next result  is concerned with determining the centralizer of a subgroup of $Q$ in $P$.

	\begin{prop}{\label{centralizers_of_a_subgroup_of_Q}}
		Let $G = P \rtimes Q$. Then
		
		\begin{enumerate}[$(\rm i)$]
			\item $P = M_0 \times  M_1 \times \cdots \times M_k$ where $Q$ acts on $M_0$ trivially and $M_i$ for $ 1 \leq i \leq k$ are irreducible ${\mathbb{Z}}_pQ$-modules.
	        \item  $dim(M_1) = dim(M_i) = d$ where $d = {\rm ord}(p\bmod{q})$ for all $1 \leq i \leq k$.
	        \item Let ${\cal B}_i$ be the basis of $M_i$ and let ${\cal B} = \{a_1,\ldots,a_{\alpha}\}$ be the set $ \cup_{i=0}^{k}{\cal B}_i$. Let $x \in Q$. Then $C_{P}(x) = span_{i \in {\cal I}}\{a_i\}$ where ${\cal I}$ is the set of indices $i$ such that ${a_i}^{x} = a_i$. Further if $\hat{Q} \leq Q$ then  there exists a subset $\{i_1,\ldots,i_r\}$ of $\{1,\ldots,k$\} such that $C_P(\hat{Q}) = M_0 \times  M_{i_1} \times \cdots \times M_{i_r}$. 
	        \item If $\hat{Q} \leq Q$ then  there exists a subset $\{i_1,\ldots,i_r\}$ of $\{1,\ldots,k$\} such that $N_G(\hat{Q}) = (M_0 \times  M_{i_1} \times \cdots \times M_{i_r}) \rtimes Q$. 
	        
		\end{enumerate}
	\end{prop}
	
	\begin{proof}
		Clearly $(\rm i)$ follows from Maschke's Theorem and $(\rm ii)$ from \cite[Theorem 2.3.2]{S1992}.
		
		\medskip 
		Now clearly $span_{i \in {\cal I}}\{a_i\}$ is contained in $C_P(x)$. Let  $b \in C_P(x)$. As each $M_i$ is a ${\mathbb{Z}_p\langle x \rangle}$-module therefore by \cite[Theorem 2.3.2]{S1992}, for each $i$ either $M_i$ is non-trivial irreducible ${\mathbb{Z}_p\langle x \rangle}$-module or $x$ acts trivially on $M_i$. Without loss of generality we can assume that $x$ acts trivially on $M_0,\ldots, M_l$ where $0 \leq l \leq k$. Then $M_0 \times \cdots \times M_l \leq span_{i \in {\cal I}}\{a_i\}$. We can write $b = b_0 \cdots b_lb_{l+1}\cdots b_k$ where $b_i \in M_i$. Then $x$ acts trivially on $b_i$ for each $0 \leq i \leq k$. So if $b_i \not = 1$ for $l+1 \leq i \leq k$, then $\langle b \rangle$ becomes a ${\mathbb{Z}_p\langle x \rangle}$-submodule of $M_i$ a contradiction. Thus $b =  b_0 \cdots b_l \in span_{i \in {\cal I}}\{a_i\}$. It also shows that $C_P(x) = M_0 \times M_1 \times \cdots \times M_l$. In Particular if $\hat{Q} \leq Q$, then there exists a subset $\{i_1,\ldots,i_r\}$ of $\{1,\ldots,k\}$ such that  $C_P(\hat{Q}) = M_0 \times M_{i_1}\times \cdots \times M_{i_r}$.
		
		\medskip
		Clearly $(\rm iv)$ follows from $(\rm iii)$ and  Proposition \ref{Complete_conjugates}.
	 \end{proof}
	 
	 So the Proposition \ref{centralizers_of_a_subgroup_of_Q} provides us a method of determining the centralizers of the subgroups of $Q$ in $P$. We now present a remark on determining the centralizers of elements of $Q$ in $P$ using their matrix representations.
	 
	 \begin{rem}
	 	Let $G = P \rtimes Q$. Then the elements of $Q$ can be realized as matrices in ${\rm GL}(\alpha,p)$ and the elements of $P$ can be realized as vectors on which $Q$ via multiplication of the matrices in ${\rm GL}(\alpha,p)$ with the vectors in $P$	. By Maschke's Theorem we have $P = M_0 \times M_1 \times \cdots \times M_k$ where $M_0$ is a trivial $\mathbb{Z}_pQ$-module and $M_i$ for $1 \leq i \leq k$ are irreducible $\mathbb{Z}_pQ$-modules. Let ${\cal B}_i$ be the basis of $M_i$. Then $\cup_{i=0}^{k}{\cal B}_i$ forms basis of $P$. Let $[a_{ij}]$ be the matrix of $x \in Q$ with respect to the basis ${\cal B}$. Let ${\cal B} = \{b_1,\ldots,b_{\alpha}\}$. Then ${b_i}^x = b_i$ if and only if the $a_{ii} = 1$. Thus the set ${\cal I} = \{i \mid {b_i}^x = b_i\}$ is equal to $\{i \mid a_{ii} = 1\}$. This standard observation allows us to detemine the centralizers of the elements of $Q$ simply by looking at the matrices of $Q$ with respect to the basis ${\cal B}$. Further let $\hat{Q} \leq Q$ and let $\hat{Q} = \langle x_1 \rangle \times \cdots \times \langle x_r \rangle$. Let ${\cal I}_j = \{i \mid {b_i}^{x_j} = b_i\}$ for $1 \leq j \leq r$. Then $C_P(\hat{Q}) = span_{i \in {\cal I}}\{b_i\}$ where ${\cal I} = \cap_{i=1}^{r}{\cal I}_i$. 
	 \end{rem}

	Before moving further we recall some standard definitions of module theory. 
	
	\begin{definition}
		An $R$ module $M$ is called isotypic if $M$ is isomorphic to a direct sum of isomorphic $R$-modules. Let $M$ be isotypic. Then we say $M$ is of type $V$ if each summand in the direct sum of $M$ is isomorphic to $V$. An $R$-submodule $N$ of $M$ isomorphic to $V$  is called a component of $M$.
	\end{definition}

 Now we provide the description of the normal subgroups of $G = P \rtimes Q$.
	
	\begin{thm}{\label{normal_subgroups_in_the_variety}}
		Let $P_1,\ldots ,P_r$ be non-isomorphic irreducible $\mathbb{Z}_pQ$-modules  and let $P = U_1 \times \cdots \times U_r$ where $U_i$'s are isotypic  ${\mathbb Z}_pQ$-submodules  of type $P_i$ of multiplicity $m_i$ for $1 \leq i \leq r$. Let $G = P \rtimes_{\theta} Q$ and let $S$ be a normal subgroup of $G$. Then one of the following holds.

		\begin{enumerate}[$(\rm i)$]
			\item $S \leq P$ is an $\mathbb{Z}_pQ$-submodule of $P$. Further there exists a subset $\{i_1,\ldots,i_s\}$ of $\{1,\ldots,r\}$ such that $S = M_{i_1} \times \cdots \times M_{i_s}$ where $M_{i_j} \leq {U_{i_j}}$ is a direct sum of isomorphic ${\mathbb{Z}_p}Q$-submodules of type $P_{i_j}$. 
	        \item $S \leq Ker(\theta)$.
	        \item $S = \hat{P} \rtimes \hat{Q}$ where $\hat{P} \leq P$ is a $\mathbb{Z}_pQ$-submodule of $P$ with $\hat{P}$ contains a complement of $C_P(\hat{Q})$ in $P$ and $\hat{Q}$ is a subgroup of $Q$. Further there exists a subset $\{l_1,\ldots,l_m\}$ of $\{1,\ldots,r\}$ such that $\hat{P} = M_{l_1} \times \cdots \times M_{l_m}$ where $M_{l_k} \leq {U_{l_k}}$ is a direct sum of isomorphic ${\mathbb{Z}_p}Q$-submodules of type $P_{l_k}$. 
	    \end{enumerate}
	    Conversely, the subgroups of $G$ described in $(\rm i)$, $(\rm ii)$ and $(\rm iii)$ are normal in $G$.
	\end{thm}
	
	\begin{proof}
		In order to prove this result we use Theorem \ref{normal_subgroup_of_a_split_extension}. If $S \leq P$, then is a $\mathbb{Z}_pQ$-submodule. Thus $S = M_{i_i}\times \cdots \times M_{i_k}$ where $M_{i_j}$ is isotypic of type $P_{i_j}$. In order to show that $M_{i_j} \leq U_{i_j}$ it is enough to show that $U_{i_j}$ is the generated by all irreducible $\mathbb{Z}_pQ$-submodules isomorphic to $P_{i_j}$. Let $U = U_{i_j}$ and let $M$ be the $\mathbb{Z}_pQ$-submodule of $P$ generated by all irreducible $\mathbb{Z}_pQ$-submodules of $P$ of type $P_{i_j}$. Then by Theorem \cite[3.3.11]{R1995}, $M$ is the direct sum of irreducible modules isomorphic to $P_{i_j}$. Further by Theorem \cite[3.3.12]{R1995} $P = M \times (P_{k_1} \times \cdots \times P_{k_u})$. But then by Lemma \cite[3.3.10]{R1995} we must have $U = M$ and the result follows.

		
		\medskip
		If $S \leq Q$ then $S \leq C_Q(P) = ker(\theta)$.
		
		\medskip
	 Finally if  $S = \hat{P} \rtimes \hat{Q}$ where $\hat{P} \leq P$, $\hat{Q} \leq Q$ and $\hat{P}$ contains a transversal of $C_P(\hat{Q})$ in $P$.  Then we can write $P = C_P(\hat{Q})\hat{P}$. Since $\hat{P}$ is elementary abelian we have $\hat{P} = C_P(\hat{Q})\cap \hat{P} \times K$ where $K \leq \hat{P}$. Thus $P = C_P(\hat{Q}) \times  K$ where $K \leq \hat{P}$ and therefore $\hat{P}$ contains a complement of $C_P(\hat{Q})$ in $P$. Since $\hat{P}$ is a $\mathbb{Z}_pQ$-module it can be shown that $\hat{P} = M_{l_1}\times \cdots \times M_{l_m}$ as in part $(\rm i)$. 
	 
	 \medskip
	 Conversely the subgroups described in $(\rm i)$, $(\rm ii)$ and $(\rm iii)$ are clearly normal in $G$.
	\end{proof}
	
	In particular if $m_i = 1$ for $1\leq i \leq r$ in the Theorem \ref{normal_subgroups_in_the_variety} then we have the following corollary.
	
	\begin{cor}
		Let $G = P \rtimes_{\theta} Q$ and  $P = {P_1} \times \cdots \times {P_r}$ where $P_i$'s are irreducible non-isomorphic ${\mathbb Z}_pQ$-module for $1 \leq i \leq r$. Let $S$ be a normal subgroup of $G$. Then one of the following holds.

		\begin{enumerate}[$(\rm i)$]
			\item $S \leq P$ is an $\mathbb{Z}_pQ$-submodule of $P$. Further there exists a subset $\{i_1,\ldots,i_s\}$ of $\{1,\ldots,r\}$ such that $S = P_{i_1} \times \cdots \times P_{i_s}$.
			\item $S \leq Ker(\theta)$.
			\item $S = \hat{P} \rtimes \hat{Q}$ where $\hat{P} \leq P$ is a $\mathbb{Z}_pQ$-submodule of $P$ with $\hat{P}$ contains a complement of $C_P(\hat{Q})$ in $P$ and $\hat{Q}$ is a subgroup of $Q$. Further there exists a subset $\{l_1,\ldots,l_m\}$ of $\{1,\ldots,r\}$ such that $\hat{P} = P_{l_1} \times \cdots \times P_{l_m}$. 
		\end{enumerate}
	\end{cor}
	
	Another particular case of Theorem \ref{normal_subgroups_in_the_variety} is when $q \mid p-1$. In this case the irreducible $\mathbb{Z}_pQ$-modules have degree $1$ and normal subgroups of $G = P \rtimes Q$ have the following description. 
	
	\begin{cor}
			Let $q \mid p-1$. Let $P_1,\ldots ,P_r$ be non-isomorphic irreducible $\mathbb{Z}_pQ$-modules  and let $P = U_1 \times \cdots \times U_r$ where $U_i$'s are isotypic  ${\mathbb Z}_pQ$-submodules  of type $P_i$ and multiplicity $m_i$ for $1 \leq i \leq r$. Let $G = P \rtimes_{\theta} Q$ and let $S$ be a normal subgroup of $G$. Then one of the following holds.

		\begin{enumerate}[$(\rm i)$]
			\item $S \leq P$ is an $\mathbb{Z}_pQ$-submodule of $P$. Further there exists a subset $\{i_1,\ldots,i_s\}$ of $\{1,\ldots,r\}$ such that $S = M_{i_1} \times \cdots \times M_{i_s}$ where $M_{i_j}$ is a subspace of $U_{i_j}$. 
			\item $S \leq Ker(\theta)$.
			\item $S = \hat{P} \rtimes \hat{Q}$ where $\hat{P} \leq P$ is a $\mathbb{Z}_pQ$-submodule of $P$ with $\hat{P}$ contains a complement of $C_P(\hat{Q})$ in $P$ and $\hat{Q}$ is a subgroup of $Q$. Further there exists a subset $\{l_1,\ldots,l_m\}$ of $\{1,\ldots,r\}$ such that $\hat{P} = W_{l_1} \times \cdots \times W_{l_m}$ where $W_{l_k}$ is a subspace of $ {U_{l_k}}$. 
		\end{enumerate}
		Conversely the subgroups described in $(\rm i)$, $(\rm ii)$ and $(\rm iii)$ are normal in $G$.
	\end{cor}
	
	\begin{proof}
		By virtue of Theorem \ref{normal_subgroups_in_the_variety} the result follows plainly because the image of $Q$ in ${\rm GL}(m_i,p)$ consists of scalar matrices only and therefore every subspace of $U_i$ is a ${\mathbb{Z}_p}Q$-module.
	\end{proof}
	
   From Theorem \ref{normal_subgroups_in_the_variety} it follows that the normal subgroups of a group $G = P \rtimes Q$ are determined by the ${\mathbb{Z}}_pQ$-submodules of an isotypic ${\mathbb{Z}}_pQ$-module. Our next objective is to describe the submodules of an isotopic ${\mathbb{Z}}_pQ$-module.  
    
    \begin{prop}
    	Let $M$ be an irreducible ${\mathbb{Z}}_pQ$-module and let $U = U_1 \times \cdots \times U_k$ be an isotypic ${\mathbb{Z}}_pQ$-module of type $M$. Let  $d = {\rm ord}(p\bmod{q})$ and let $\mathbb{F}$ be the field of order $p^d$. Then 
    	\begin{enumerate}[$(\rm i)$]
    		
    		\item if $V$ is an irreducible submodule of $U$, then for $\{1\} \not = v \in V$ there exists $\gamma \in {\rm GL}(U)$ such that the set $\{v,\gamma^p(v),\ldots,\gamma^{p^{d-1}}(v)\}$
    		\item there exists a subset $\{i_1,\ldots,i_r\}$ such that V is the subdirect product of $\{U_{i_1}, \ldots,U_{i_r}\}$.
    		\item  there is a one to one correspondence between the irreducible $\mathbb{Z}_pQ$-submodules of $U$ and one dimensional $\mathbb{F}$-subspaces of $\mathbb{F}^k$.
    	\end{enumerate}
    \end{prop}
    
    \begin{proof}
    	Let $\theta: Q \rightarrow {\rm GL}(U)$ be the homomorphism arises from the action of $Q$ on $U$. Then the image of $Q$ under $\theta$ is contained in the group ${\rm GL}(U_1) \times \cdots \times {\rm GL}(U_k)$. Let $\pi_i : \theta(Q) \rightarrow {\rm GL}(U_i)$ be the projection map and let $\theta_i = \pi_i \circ \theta: Q \rightarrow {\rm GL}(U_i)$. Then the image of $Q$ under $\theta_i$ is cyclic. Since $U_i \cong U_j$ for all $1 \leq i,j \leq k$ we have that $ker(\theta_i) = ker(\theta_j)$ for all $1 \leq i,j\leq k$. Suppose that $x \in Q$ acts non-trivially on $U_1$. Then $x$ acts non-trivially on $U_i$ for all $1 \leq i \leq k$. Let $\theta(x) = \gamma \in {\rm GL}(U)$.  Now if the set $\{v,\gamma^p(v),\ldots,\gamma^{p^{d-1}}(v)\}$ is linearly dependent then the vector space $span\{v,\gamma^p(v),\ldots,\gamma^{p^{d-1}}(v)\}$ becomes a proper ${\mathbb{Z}_p}Q$-submodule of $V$ a contradiction. 
    	
    	\medskip
    	Suppose that $V = \{v,\gamma^p(v),\ldots,\gamma^{p^{d-1}}(v)\}$ for some $\gamma \in {\rm GL}(U)$. Let $\eta_i: U \rightarrow U_i$ be the natural projection. Let $\{i_1,\ldots,i_r\}$ be the set such that $\eta_{i_j}(v) \not = 1$. Then $\eta_{i_j}: V \rightarrow U_{i_j}$ are $\mathbb{Z}_pQ$-module isomorphisms for $1 \leq j \leq r$. Since the set $\{v,\gamma^p(v),\ldots,\gamma^{p^{d-1}}(v)\}$ form basis for $V$ and $\eta_i$ for all $1 \leq i \leq k$ commute with $\theta$, the projection $\eta_j: V \rightarrow U_j$ is trivial for $j \in \{1,\ldots,k\} \setminus \{i_1,\ldots,i_r\}$. Therefore $V$ is the subdirect product of $U_{i_1},\ldots,U_{i_r}$.
    	
    	\medskip
    	Last part of this result follows from \cite[Theorem 27.14]{A1986} 
    \end{proof} 
    
	Let $G = P \rtimes Q$. There are two potential approaches to construct normal subgroups of $G$ of the form $\hat{P}\rtimes \hat{Q}$ where $\hat{P}\leq P$ and $\hat{Q}\leq Q$. First is to fix a subgroup $\hat{P}$ of $P$ and subsequently find a suitable subgroup $\hat{Q}$ of $Q$. Alternatively one can fist fix a subgroup $\hat{Q}\leq Q$ and then determine a suitable subgroup $\hat{P}$ of $P$. We incline towards the second approach as it aligns more naturally with the framework of Theorem \ref{normal_subgroups_in_the_variety}. 
	
	\medskip 
	A standard procedure for constructing a normal subgroup of $G$ of the form $\hat{P}\rtimes \hat{Q}$ is summarized in the following remark.
	
	\begin{rem}
		Let $\hat{Q}$ be a subgroup of $Q$. The matrices representing elements of $Q$ can be obtained from \cite[Theorem 2.1]{BKV} and consequently can be used to represent the matrices of elements of $\hat{Q}$. Proposition \ref{centralizers_of_a_subgroup_of_Q} then allows to compute the centralizer $C_P(\hat{Q})$ of $\hat{Q}$ in $P$. Furthermore by \cite[Theorem 27.14]{A1986} one can identify the $\mathbb{Z}_pQ$-submodules of $P$ and from this we can choose a $\mathbb{Z}_pQ$-submodule $\hat{P}$ of $P$ that contains a complement of $C_P(\hat{Q})$. By Theorem \ref{normal_subgroups_in_the_variety} the resulting subgroup $\hat{P}\rtimes \hat{Q}$ is normal in $G$. 
	\end{rem}

	
	
	We conclude this paper by providing the complete classification of normal subgroups of a finite group all of whose proper subgroups are nilpotent. Classification of such groups is due to O. J. Schmidt (see \cite{S1924}). Proof is straight forward and follows directly from Theorem \ref{normal_subgroup_of_a_split_extension}.
	
	\begin{thm}{\label{Normal_subgroups_of_a_group_whose_all_subgroups_are_nipotent}}
		Let $G = N \rtimes K$ where $|N| = p^{\alpha}$, $|K|=q^{\beta}$ and $K$ is cyclic be a finite minimal non-nilpotent group. Let $S$ be a proper normal subgroup of $G$. Then one of the following holds.
		
		\begin{enumerate}[$(\rm i)$]
			\item Either $S = N$ or $S \leq Z(N)$.
			\item $S \leq \Phi(K)$.
			\item $S = \hat{N} \times \hat{K}$ with $\hat{N} = N$ or $\hat{N} \leq Z(N)$ and $\hat{K} \leq \Phi(K)$.

		\end{enumerate}
		Conversely the subgroups of $G$ described in $(\rm i)$, $(\rm ii)$ and $(\rm iii)$ are normal in $G$. 
	\end{thm}

	\section*{Acknowledgment}
	The author would like to thank {\emph{Prof Geetha Venkataraman}} and her student and my research fellow {\emph{Dr. Arushi}} for suggesting this problem and encouragement during the preparation of this work.

    \end{document}